\documentclass[smallheadings]{scrartcl}

\pdfoutput=1

\usepackage[T1]{fontenc}
\usepackage[latin1]{inputenc}

\usepackage{amssymb}
\usepackage{amsfonts}
\usepackage{amsmath}
\usepackage{wasysym}

\usepackage{epsfig}
\usepackage{epic}
\usepackage{eepic}
\usepackage{graphics}
\usepackage{graphicx}
\usepackage{color}
\usepackage{longtable}

\usepackage{amsthm}

\newcommand{\ds}{\displaystyle}
\newcommand{\ts}{\textstyle}

\newcommand{\var}{\operatorname{var}}

\newcommand{\Nn}{{\mathbb N}}

\newcommand{\Rr}{{\mathbb R}}
\newcommand{\Cc}{{\mathbb C}}

\newcommand{\eps}{\varepsilon}

\newcommand{\Jj}{\mathbf{J}}
\newcommand{\ccc}{\mathbf{c}}

\newcommand{\Loc}{\mathcal{L}}

\newcommand{\Rcal}{\mathcal{R}}

\newcommand{\Sd}{{ \mathbb S}}

\newcommand{\pia}{\Pi_n}
\newcommand{\pib}{\Pi_{n}^m}
\newcommand{\pic}{\Pi_n^{\mathcal{R}}}
\newcommand{\usa}{{\mathbb S}_n}
\newcommand{\usb}{{\mathbb S}_{n}^m}
\newcommand{\usc}{{\mathbb S}_n^{\mathcal{R}}}
\newcommand{\Ppa}{\mathcal{P}_n}
\newcommand{\Ppb}{\mathcal{P}_n^m}
\newcommand{\Ppc}{\mathcal{P}_n^{\mathcal{R}}}

\newcommand{\Hha}{\mathcal{H}_n}
\newcommand{\Hhb}{\mathcal{H}_n^m}

\def\be{\begin{equation}}
\def\ee{\end{equation}}
\def\bean{\begin{eqnarray*}}
\def\eean{\end{eqnarray*}}

\theoremstyle{plain}
\newtheorem{Thm}{Theorem}[section]

\newtheorem{Cor}[Thm]{Corollary}
\newtheorem{Lem}[Thm]{Lemma}

\theoremstyle{remark}
\newtheorem{Rem}[Thm]{Remark}
\newtheorem{Exa}[Thm]{Example}

\theoremstyle{definition}
\newtheorem{Def}[Thm]{Definition}

\begin{document}

\title{Optimally space localized polynomials \\ with applications in signal processing}
\author{Wolfgang Erb
\thanks{Institute of Mathematics, University of Lübeck,
Wallstrasse 40, 23560 Lübeck, Germany. erb@math.uni-luebeck.de}}

\date{5.08.2010}

\maketitle
\begin{abstract}
For the filtering of peaks in periodic signals, we specify polynomial filters that are optimally localized in space. The space localization of functions $f$ having an expansion in terms of orthogonal polynomials is thereby measured by a generalized mean value $\eps(f)$. Solving an optimization problem including the functional $\eps(f)$, we determine those polynomials out of a polynomial space that are optimally localized. We give explicit formulas for these optimally space localized polynomials and determine in the case of the Jacobi polynomials the relation of the functional $\eps(f)$ to the position variance of a well-known uncertainty principle. Further, we will consider the Hermite polynomials as an example on how to get optimally space localized polynomials in a non-compact setting. Finally, we investigate how the obtained optimal polynomials can be applied as filters in signal processing.
\end{abstract}

\section{Introduction}

For the detection of peaks in mass spectrometry data, window functions are often used to preprocess the incoming signals,
for instance, to perform a baseline correction or to filter out disturbing higher frequencies (cf. \cite{Nguyen2006}, \cite{Yang2009}).
In the following, we consider $2\pi$-periodic signals $f$ filtered by a convolution with a window function $h$, i.e.
the filtered signal $F_h f$ is given by
\begin{equation}
\label{equation-filteroperator} F_h f(t) := (f \ast h)(t) = \frac{1}{2\pi}\int_{-\pi}^{\pi} f(s) h(t-s) ds.
\end{equation}
For peak detection purposes, two properties of the window function $h$ are important. First of all,
the function $h$ should be localized in frequency in order to filter out the high frequencies. On the other hand, the window $h$ should also be well
localized in space such that the convolution operator $F_h$ is still able to determine the peaks of the signal $f$. However, the uncertainty principle states that it is impossible that the function $h$ is well-localized both in space and frequency
(see, for instance, \cite{Erb2010}, \cite{FollandSitaram1997}, \cite{Gröchenig2003}). Therefore, in search for an optimal filter $h$ for 
peak detection, one has always a trade off between denoising the signal $f$ and determining the position of the peaks of $f$.

The main objective of this article is to investigate the space localization of polynomial filters. If the window function $h$ is a trigonometric polynomial of degree $n$, the frequency domain of $h$ is bounded. Therefore, a polynomial filter $h$ in a peak detection process automatically performs a low-pass filtering of the signal $f$. It remains to analyze the space localization of $h$. In this regard, one has to specify how space localization of a function $f$ is measured.

Beside the trigonometric setting, we will consider general systems of orthogonal polynomials in this article. For functions $f$ having an expansion in orthogonal polynomials, we will use a functional $\eps(f)$ to measure the space localization of $f$. We will study an optimization problem including this functional $\eps(f)$ that allows
us to construct polynomials and band-limited functions that are optimally localized in the space domain.
In particular, for Jacobi polynomials, it will turn out that the functional $\eps(f)$ is related to the position variance $\var_S(f)$ of a well-known uncertainty principle and that the optimally space localized polynomials
minimize also the term for the position variance $\var_S(f)$.

By now, research in this direction has been done mainly for trigonometric polynomials on the unit circle and
spherical harmonics on the unit sphere $\Sd^d$. In \cite{Rauhut2005}, Rauhut used the position variance of the Breitenberger uncertainty
principle to construct optimally localized polynomials on the unit circle. On the unit sphere, the works of
Mhaskar et al. \cite{MhaskarNarkovichPrestinWard2000} and
La{\'i}n Fern{\'a}ndez \cite{Fernandez2007} led to optimally space localized polynomials and polynomial wavelets on the unit sphere $\Sd^d$.
One aim of this article is to extend these results to orthogonal expansions on the interval $[-1,1]$. These more general results are then used to construct new window functions for peak filtering purposes in the trigonometric setting.

In the following, we will mainly consider the Hilbert space $L^2([-1,1],w)$ with the inner product
\[ \langle f,g \rangle_{w} = \int_{-1}^1 f(x) \overline{g(x)} w(x) dx,\]
where the weight function $w$ is a nonnegative continuous function on $[-1,1]$.
Then, for $f \in L^2([-1,1],w)$, we define the functional $\eps(f)$ by
\begin{equation} \label{equation-meanvalue}
\eps(f) := \int_{-1}^{1} x |f(x)|^2 w(x) dx.
\end{equation}

Exactly this generalized mean value $\eps(f)$ is the starting point for our investigations.
In the following first section, we will see that the mean value $\eps(f)$ determines a measure for the localization of the function $f$ on
the boundary values of the interval $[-1,1]$. Then, we will study the functional $\eps(f)$ for
polynomial subspaces of $L^2([-1,1],w)$. In particular, in Theorem \ref{Theorem-optimalpolynomial} and in Corollary
\ref{Corollary-explicitformoptimalpolynomials}, we will give representations of
those polynomials $\Ppa$ that maximize $\eps(f)$, i.e. those polynomials that are optimally localized at the right hand boundary of the interval $[-1,1]$ with
respect to $\eps(f)$.

In the third section, we will emphasize on the Jacobi polynomials. In Theorem \ref{Theorem-minequivalenttomax}, it will be shown that the position variance $\var_S(f)$ of an uncertainty principle for Jacobi expansions is minimized by the polynomials $\Ppa$.

In the fourth section, we will consider the Hermite polynomials on the real line and determine optimally space localized polynomials
in this setting. The obtained results are an example of how it is possible to generalize
the theory of Section \ref{Section-optimallylocalizedpolynomials} and \ref{Section-explicitexpressions} to a non-compact setting.

Finally, in the last section we will turn back to the peak filtering application mentioned at the beginning. We will investigate how the optimally space localized polynomials of this article are related to well-known polynomial window functions and how they can be applied as filters for the peak detection of a signal.

\section{Optimally space localized polynomials} \label{Section-optimallylocalizedpolynomials}

We start out by introducing particular polynomial subspaces of $L^2([-1,1],w)$.
Therefore, we denote by $\{p_l\}_{l=0}^\infty$ the family of polynomials that are orthonormal on $[-1,1]$ with respect to the inner product
$\langle \cdot, \cdot \rangle_w$. Further, we assume that the polynomials $p_l$ of degree $l$ are normalized such that the coefficient of $x^l$ is positive. Then,
the family $\{p_l\}_{l=0}^\infty$ defines a complete orthonormal set in the Hilbert space $L^2([-1,1],w)$ (cf. \cite[Section 2.2]{Szegö}).

\begin{Def} \label{definition-polynomialspacesJacobi}
As subspaces of the Hilbert space $L^2([-1,1],w)$, we consider the following three polynomial spaces:
\begin{enumerate}
\item[(1)] The space spanned by the polynomials $p_l$, $0 \leq l \leq n$:
\begin{equation} \pia :=  \left\{P:\;P(x)= \sum_{l=0}^n c_l p_l(x),\;c_0, \ldots, c_n \in \Cc \right\}.\end{equation}
\item[(2)] The space spanned by the polynomials $p_l$, $m \leq l \leq n$:
\begin{equation} \pib := \left\{P:\;P(x)= \sum_{l=m}^n c_l p_l(x),\;c_m, \ldots, c_n \in \Cc\right\}.\end{equation}
\item[(3)] The space spanned by a polynomial $\ds \Rcal(x) = p_m(x)+\sum_{l=0}^{m-1} e_l p_l(x)$
 of degree $m$ and the polynomials $p_l$, $m+1 \leq l \leq n$:
\begin{equation} \pic :=  \left\{P:\;P(x) = c_m \Rcal(x) + \sum_{l=m+1}^n c_l p_l(x),\;c_m, \ldots, c_n \in \Cc\right\}.\end{equation}
\end{enumerate}
Further, we define the unit spheres of the spaces $\pia$, $\pib$ and $\pic$ as
\begin{align*}
\usa &:= \left\{P \in \pia: \; \|P\|_{w} = 1\right\}, \\ \usb &:= \left\{P \in \pib: \; \|P\|_{w} = 1\right\}, \\
\usc &:= \left\{P \in \pic: \; \|P\|_{w} = 1\right\}.\end{align*}
\end{Def}

\begin{Rem}
Clearly, $\pib \subset \pia$ and $\pic \subset \pia$. In the literature,
the spaces $\pib$ are sometimes called wavelet spaces and considered in a more general theory
on polynomial wavelets and polynomial frames, see \cite{MhaskarPrestin2005} and the references therein.
For special choices of $\Rcal$, the polynomials in the spaces $\pic$ play an important role in the theory of polynomial approximation. In particular, if polynomial reproduction is requested, a common choice for the polynomial $\Rcal$ is the Christoffel-Darboux kernel of degree $m$ (see \cite{FilbirMhaskarPrestin2009}, \cite{Mhaskar}). The standardization $e_m = 1$ for the highest expansion coefficient of the polynomial $\Rcal$ causes no loss of generality and is a useful convention for the upcoming calculations.
\end{Rem}

The first goal of this section is to study the localization of the polynomials in the spaces $\pia$, $\pib$ and $\pic$ at the right hand boundary of the interval $[-1,1]$ and to determine those polynomials that are in some sense best localized.
As an analyzing tool for the localization of a function $f \in L^2([-1,1])$ at the point $x = 1$, we consider the mean value $\eps(f)$
as defined in \eqref{equation-meanvalue}.
If $\|f\|_{w} = 1$, then $-1 < \eps(f) < 1$, and the more the mass of the $L^2$-density
$f$ is concentrated at the boundary point $x=1$, the closer the value $\eps(f)$
gets to $1$. Therefore, the value
$\eps(f)$ can be interpreted as a measure on how well the function $f$ is localized at the right hand boundary of the interval $[-1,1]$.
We say that $f$ is localized at $x=1$ if the value $\eps(f)$ approaches $1$.

Now, our aim is to find those elements of the polynomial spaces $\pia$, $\pib$ and $\pic$ that
are optimally localized at the boundary point $x=1$. In particular, we want to solve the following optimization problems:
\begin{align}
 \Ppa &= \arg\max_{P \in \usa} \eps(P),  \label{optimalpolynomiala}\\
 \Ppb &= \arg\max_{P \in \usb} \eps(P),  \label{optimalpolynomialb}\\
 \Ppc &= \arg\max_{P \in \usc} \eps(P).  \label{optimalpolynomialc}
\end{align}

\begin{Rem}
Since the linear spaces $\pia$, $\pib$ and $\pic$ are finite-dimensional, the unit spheres $\usa$, $\usb$ and $\usc$
are compact subsets
and the functional $\eps$ is bounded and continuous on the respective polynomial space. Hence,
it is guaranteed that solutions of the optimization problems \eqref{optimalpolynomiala}, \eqref{optimalpolynomialb} and \eqref{optimalpolynomialc} exist.
\end{Rem}

\begin{Rem}
The optimization problem \eqref{optimalpolynomiala} has a well-known solution which can be found in \cite[Theorem 1.3.3]{Mhaskar}.
The solutions for \eqref{optimalpolynomialb} and \eqref{optimalpolynomialc} given below can be considered as novel.
\end{Rem}

\begin{Rem}
The functional $\eps(f)$ is by far not the only possible way to measure the space localization of a function. In the literature, there exist various other forms in this direction. In the Landau-Pollak-Slepian theory (cf. \cite{FollandSitaram1997}, \cite{Landau1985},
\cite{Slepian1983}) an optimization problem similar as in \eqref{optimalpolynomiala} is investigated.
Results in terms of polynomial and exponential growth of polynomials can be found in the articles \cite{FilbirMhaskarPrestin2009} and \cite{IvanovPetrushevXu2010}. Results concerning the 
Shannon information entropy of orthogonal polynomials are summarized in the survey article \cite{AptekarevDehesaMartinezFinkelshtein2010}.
\end{Rem}

In order to describe the optimal polynomials, we need the notion of associated and of scaled co-recursive associated polynomials.
First of all, we know that the orthonormal polynomials $p_l$ satisfy the following three-term recurrence relation
(cf. \cite[Section 1.3.2]{Gautschi})
\begin{align} \label{equation-recursionorthonormal}
b_{l+1} p_{l+1}(x) &= (x - a_l) p_l(x) - b_l p_l(x), \quad l=0,1,2,3, \ldots \\
 p_{-1}(x) &= 0, \qquad p_0(x) = \frac{1}{b_0}, \notag
\end{align}
with coefficients $a_l \in \Rr$ and $b_l > 0$.

\begin{Def} \label{definition-associatedJacobi}
For $m \in \Nn$, the associated polynomials $p_l(x,m)$ on the interval $[-1,1]$ are defined by the shifted recurrence relation
\begin{align} \label{equation-recursionassociatedsymmetric}
b_{m+l+1}\, p_{l+1}(x,m) &= (x - a_{m+l})\, p_l(x,m) - b_{m+l}\, p_{l-1}(x,m),
\quad l=0,1,2, \ldots , \\
 p_{-1}(x,m) &= 0, \qquad p_0(x,m) = 1. \notag
\end{align}
Further, for $\gamma \in \Rr$ and $\delta \geq 0$, we define the scaled co-recursive associated polynomials
$p_l(x,m,\gamma,\delta)$ on $[-1,1]$ by the three-term recurrence relation
\begin{align} \label{equation-recursionassociatedscaledsymmetric}
b_{m+l+1}\, p_{l+1}(x,m,\gamma,\delta) &= (x - a_{m+l})\, p_l(x,m,\gamma,\delta) - b_{m+l}\,
p_{l-1}(x,m,\gamma,\delta), \notag\\
 & \qquad l =1,2,3,4 \ldots , \\
p_0(x,m,\gamma,\delta) &= 1, \quad p_{1}(x,m,\gamma,\delta) =
\frac{\delta x - a_m - \gamma}{\beta_{m+1}}. \notag
\end{align}
\end{Def}

The three-term recurrence relation of the co-recursive associated polynomials $p_{l+1}(x,m,\gamma,\delta)$
corresponds to the three-term recurrence relation of the associated polynomials $p_{l+1}(x,m)$ except for the formula of the initial polynomial
$p_{1}(x,m,\gamma,\delta)$.
For $c = 0$, $\gamma = 0$ and $\delta = 1$, we have the identities $p_l(x,0) = p_l(x,0,0,1) = b_0\, p_l(x)$.
The polynomials $p_l(x,m)$ and $p_l(x,m, \gamma, \delta)$
can be described with help of the symmetric Jacobi matrix $\Jj_n^m$, $0 \leq m \leq n $, defined by
\begin{equation} \label{equation-Jacobimatrix}
\Jj_n^m = \left(\begin{array}{cccccc}
a_m & b_{m+1} &  0 &  0 & \cdots & 0 \\
b_{m+1} &  a_{m+1} &  b_{m+2} &  0 & \cdots & 0 \\
0 &  b_{m+2}  & a_{m+2}  & b_{m+3} &  \ddots & \vdots \\
\vdots & \ddots & \ddots & \ddots & \ddots & 0\\
0 &  \cdots &  0&  b_{n-2} & a_{n-1} & b_{n-1} \\
0 &  \cdots &  \cdots & 0 & b_{n-1}& a_n
\end{array}\right).\end{equation}
If $m=0$, we write $\Jj_n$ instead of $\Jj_n^0$. Then, in view of the three-term recurrence formulas
\eqref{equation-recursionassociatedsymmetric} and \eqref{equation-recursionassociatedscaledsymmetric}, the polynomials $p_l(x,m)$
and $p_l(x,m,\gamma, \delta)$, $l \geq 1$, can be written as (cf. \cite[Theorem 2.2.4]{Ismail})
\begin{equation}
p_l(x,m) = \det(x \mathbf{1}_{l} - \Jj_{m+l-1}^m ), \label{equation-relation3termJacobimatrixassociated}
\end{equation}
and
\begin{equation}
p_l(x,m,\gamma,\delta) = \det \left(
x \left(\begin{array}{cc}
\delta &   0 \\
0 & \mathbf{1}_{l-1}
\end{array}\right)  - \Jj_{m+l-1}^m - \left(\begin{array}{cc}
\gamma &   0 \\
0 &  \mathbf{0}_{l-1}
\end{array}\right)\right), \label{equation-relation3termJacobimatrixscaled}
\end{equation}
where $\mathbf{1}_{l-1}$ denotes the $(l-1)$-dimensional identity matrix and $\mathbf{0}_{l-1}$ the $(l-1)$-dimensional zero matrix.

Next, we give a characterization of the functional $\eps(P)$
in terms of the expansion coefficients $c_l$ of the polynomial $P = \sum_{l=n}^m c_l p_l$.

\begin{Lem} \label{Lemma-characterizationofepsP}
For the polynomial $P(x) = \ds \sum_{l=0}^n c_l p_l(x)$, we have
\begin{align*}
\eps(P) &= \ccc^H \Jj_{n} \ccc,  & \text{if} \quad P \in \pia, \\
\eps(P) &= \tilde{\ccc}^H \Jj_{n}^m \tilde{\ccc}, & \text{if} \quad P \in \pib, \\
\eps(P) &= \tilde{\ccc}^H \Jj_{n}^m \tilde{\ccc} + (\eps(\Rcal)-a_m) |c_m|^2, & \text{if} \quad P \in \pic,
\end{align*}
with the coefficient vectors $\ccc = (c_0, \ldots, c_n)^T$ and $\tilde{\ccc} = (c_m, \ldots, c_n)^T$.
\end{Lem}

\begin{proof}
Using the three-term recurrence formula (\ref{equation-recursionorthonormal}) and the orthonormality relation of the polynomials $p_l$,
we get for $P \in \pia$
\begin{align*}
\eps(P) &= \int_{-1}^1 x \Big|\sum_{l=0}^n c_l p_l(x)\Big|^2 w(x) dx =\int_{-1}^1 \Big(\sum_{l=0}^n c_l x p_l(x)\Big)
\overline{\Big(\sum_{l=0}^n c_l p_l(x)\Big)}w(x) dx \\
&= \int_{-1}^1 \Big(\sum_{l=0}^n c_l \big(b_{l+1} p_{l+1}(x)
+ a_l p_l(x)+ b_{l} p_{l-1}(x) \big)\Big)
\overline{\Big(\sum_{l=0}^n c_l p_l(x)\Big)}w(x) dx  \\
 &= \sum_{l=0}^n a_l |c_l|^2 + \sum_{l=0}^{n-1}(b_{l+1} c_l \bar{c}_{l+1} + b_{l+1} \bar{c}_l c_{l+1}) = \ccc^H \Jj_{n} \ccc.
\end{align*}
If $c_0 = \ldots = c_{m-1} = 0$, we get the assertion for polynomials $P$ in the space $\pib$. If $P \in \pic$, then $P$ has
the representation
\[P(x) = c_m \left( p_m(x)+\sum_{l=0}^{m-1} e_l p_l(x) \right)+ \sum_{l=m+1}^{n} c_l p_l(x),\]
where the polynomial $\Rcal$ is given by $\Rcal(x) = p_m(x)+\sum_{l=0}^{m-1} e_l p_l(x)$.
Inserting this representation in the upper formula for $\eps(P)$, yields the identity
$\eps(P) = (\eps(\Rcal)-a_m)|c_m|^2 + \tilde{\ccc}^H \Jj_{n}^m \tilde{\ccc}$.
\end{proof}

Using the characterization of $\eps(P)$ in Lemma \ref{Lemma-characterizationofepsP}, we proceed to the solution of
the optimization problems \eqref{optimalpolynomiala}, \eqref{optimalpolynomialb} and \eqref{optimalpolynomialc}.

\begin{Thm} \label{Theorem-optimalpolynomial}
The solutions of the optimization problems \eqref{optimalpolynomiala}, \eqref{optimalpolynomialb} and \eqref{optimalpolynomialc} are given by
\begin{align}
\Ppa (x) &= \kappa_1 \sum_{l=0}^n p_l(\lambda_{n+1})\, p_l(x),
\displaybreak[0] \label{equation-optimalpolynomialJacobia}\\
\Ppb (x) &= \kappa_2 \sum_{l=m}^n p_{l-m}(\lambda_{n-m+1}^m,m)\, p_l(x),
\displaybreak[0] \label{equation-optimalpolynomialJacobib}\\
\Ppc (x) &= \kappa_3 \left( \Rcal(x) +
\sum_{l=m+1}^n p_{l-m}(\lambda_{n-m+1}^{\Rcal},m,\gamma_{\Rcal},\delta_{\Rcal})
p_l(x)\right), \label{equation-optimalpolynomialJacobic}
\end{align}
where $p_l(x,m)$ and $p_l(x,m,\gamma_{\Rcal},\delta_{\Rcal})$ denote the associated and the
scaled co-recursive associated polynomials as given in Definition \ref{definition-associatedJacobi} with the shift term
$\gamma_{\Rcal} := \eps(\Rcal)-a_m$ and the scaling factor $\delta_{\Rcal} := \|\Rcal\|_{w}^2$.\\
The values $\lambda_{n+1}$, $\lambda_{n-m+1}^m$ and $\lambda_{n-m+1}^{\Rcal}$
denote the largest zero of the polynomials
$p_{n+1}(x)$, $p_{n-m+1}(x,m)$ and $p_{n-m+1}(x,m,\gamma_{\Rcal},\delta_{\Rcal})$
in the interval $[-1,1]$, respectively. The constants $\kappa_1$, $\kappa_2$ and $\kappa_3$ are chosen such that the optimal polynomials
lie in the respective unit sphere and are uniquely determined up to multiplication with a complex scalar of absolute value one.
The maximal value of $\eps$ in the respective polynomial space is given by
\begin{align*}
M_{n}  &:= \max_{P\in \usa}\eps(P) = \lambda_{n+1},\\
M_{n}^m  &:= \max_{P\in \usb}\eps(P) = \lambda_{n-m+1}^{m},\\
M_n^{\Rcal}  &:= \max_{P\in \usc}\eps(P) = \lambda_{n-m+1}^{\Rcal}.
\end{align*}
\end{Thm}

\begin{proof}
We start out by determining the optimal solution $\Ppb$ for
the optimization problem \eqref{optimalpolynomialb}. The formula for
the optimal polynomial $\Ppa$ follows as a special case if we set $m=0$. First of all, Lemma \ref{Lemma-characterizationofepsP} states that the
mean value $\eps(P)$ of a polynomial $P(t) = \sum_{l=m}^n c_l p_l(x)$
can be written as $\eps(P) = \tilde{\ccc}^H \Jj_{n}^m \tilde{\ccc}$ with the coefficient vector $\tilde{\ccc} = (c_m, \cdots, c_n)^T$.
Thus, maximizing $\eps(P)$ with respect to a normed polynomial $P \in \usb$
is equivalent to maximize the quadratic functional $\tilde{\ccc}^H \Jj_{n}^m \tilde{\ccc}$ subject to
$|\tilde{\ccc}|^2 = c_m^2 + c_{m+1}^2 + \cdots + c_n^2 = 1$.
If $\lambda_{n-m+1}^{m}$ denotes the largest eigenvalue of the symmetric Jacobi
matrix $\Jj_{n}^m$, we have
\begin{equation} \label{equation-extremalmatrix}
\tilde{\ccc}^H \Jj_{n}^m \tilde{\ccc} \leq \lambda_{n-m+1}^{m}|\tilde{\ccc}|^2
\end{equation}
and equality is attained for the eigenvectors corresponding to $\lambda_{n-m+1}^{m}$.
Now, the largest eigenvalue of the Jacobi matrix $\Jj_{n}^m$ corresponds exactly with the largest
zero of the associated polynomial $p_{n-m+1}(x,m)$ (cf. \cite[Theorem 1.31]{Gautschi}).
Using the recursion formula (\ref{equation-recursionassociatedsymmetric}) of the associated polynomials $p_{l}(x,m)$ with
$c_m = 1$ the eigenvalue equation $\Jj_n^m \tilde{\ccc} = \lambda_{n-m+1}^{m} \tilde{\ccc}$ yields
\begin{align*}
c_l = p_{l-m}(\lambda_{n-m+1}^{m},m), \quad l = m, \ldots n.
\end{align*}
Finally, we have to normalize the coefficients
$c_l$, $m \leq l \leq n$, such that $|\tilde{\ccc}|^2 =1$.
This is done by the absolute value of the constant $\kappa_2$.
The uniqueness (up to a complex scalar with absolute value $1$) of the optimal polynomial $\Ppa$
follows from the fact that the largest zero
of $p_{n-m+1}(x,m)$ is simple (see \cite[Theorem 5.3]{Chihara}).
The formula for $M_{n}^m$ follows directly from the estimate in (\ref{equation-extremalmatrix}).

We consider now the third polynomial space $\pic$. Lemma \ref{Lemma-characterizationofepsP} states that in this case the
mean value $\eps(P)$ of $P(x) = c_m \Rcal(x)+\sum_{l=m+1}^n c_l p_l(x)$
can be written as $\eps(P) = \tilde{\ccc}^H \Jj_{n}^m \tilde{\ccc} + (\eps(\Rcal)-a_m) |c_m|^2$,
with the coefficient vector $\tilde{\ccc} = (c_m, \cdots, c_n)^T$. Maximizing $\eps(P)$ with respect to a polynomial $P \in \usc$
is therefore equivalent to maximize the quadratic functional $\tilde{\ccc}^H \Jj_{n}^m \tilde{\ccc} + (\eps(\Rcal)-a_m) |c_m|^2$
subject to $(\|\Rcal\|_{w}^2-1) |c_m|^2+|\tilde{\ccc}|^2 = 1$.
Using a Lagrange multiplier $\lambda$ and differentiating the Lagrange function, we obtain the identity
\[\Jj_n^m \tilde{\ccc} + \gamma_{\Rcal} (c_m, 0 , \cdots, 0)^T = \lambda \big(\delta_{\Rcal} c_m, c_{m+1}, \cdots, c_n\big)^T \]
as a necessary condition for the maximum, where $\gamma_{\Rcal} = \eps(\Rcal)-a_m$ and
$\delta_{\Rcal} = \|\Rcal\|_{w}^2$. By the equation \eqref{equation-relation3termJacobimatrixscaled}, this system of equations
is related to the three-term recursion formula (\ref{equation-recursionassociatedscaledsymmetric})
of the scaled co-recursive associated polynomials $p_l(x,m,\gamma_{\Rcal},\delta_{\Rcal})$.
In particular, the value $\lambda$ corresponds to a root of $p_{n-m+1}(x,m,\gamma_{\Rcal},\delta_{\Rcal})$.
Moreover, the maximum of $\tilde{\ccc}^H \Jj_{n}^m \tilde{\ccc} + \gamma_{\Rcal} |c_m|^2$
is attained for the largest root $\lambda = \lambda_{n-m+1}^{\Rcal}$ of $p_{n-m+1}(x,m,\gamma_{\Rcal},\delta_{\Rcal})$ and the corresponding
eigenvector
\[\tilde{\ccc} = \kappa_3
\Big( 1, p_{1}(\lambda_{n-m+1}^{\Rcal},m,\gamma_{\Rcal},\delta_{\Rcal}),
\ldots, p_{n-m}(\lambda_{n-m+1}^{\Rcal},m,\gamma_{\Rcal},\delta_{\Rcal}) \Big)^T,\]
where the constant $\kappa_3$ is chosen such that the condition $(\delta_{\Rcal}-1) |c_m|^2+|\tilde{\ccc}|^2 = 1$ is satisfied.
The uniqueness of the polynomial $\Ppc$ (up to a complex scalar of absolute value one)
follows from the simplicity of the largest root $\lambda_{n-m+1}^{\Rcal}$ of the polynomials $p_l(x,m,\gamma_{\Rcal},\delta_{\Rcal})$
(see \cite[Theorem 5.3]{Chihara}).
From the above argumentation it is also clear that
the maximal value $M_n^{\Rcal}$ is precisely the largest eigenvalue $\lambda_{n-m+1}^{\Rcal}$.
\end{proof}

\section{Explicit expression for the optimally space localized polynomials} \label{Section-explicitexpressions}

Our next goal is to find explicit expressions for the optimal polynomials $\Ppa$, $\Ppb$ and $\Ppc$ derived in
Theorem \ref{Theorem-optimalpolynomial}. To this end, we need a
Christoffel-Darboux type formula for the associated polynomials $p_l(x,m)$ and
$p_l(x,m,\gamma,\delta)$.

\begin{Lem} \label{Lemma-ChristoffelDarbouxassociated}
Let $p_l(x,m)$ and $p_l(x,m,\gamma,\delta)$ be the associated and the scaled co-recursive associated
polynomials as defined in \eqref{equation-recursionassociatedsymmetric} and \eqref{equation-recursionassociatedscaledsymmetric}. Then, for $x \neq y$, 
the following Christoffel-Darboux type formulas hold:
\begin{align} \label{equation-ChristoffelDarbouxassociated}
\sum_{k=m}^n  & p_k(x)  p_{k-m}(y,m)  \\ &= b_{n+1}\frac{p_{n+1}(x)
p_{n-m}(y,m)-p_{n-m+1}(y,m)p_n(x)}{x-y}
+ b_m \frac{p_{m-1}(x)}{x-y}, \notag \displaybreak[0] \\
\sum_{k=m}^n  & p_k(x)  p_{k-m}(y,m,\gamma,\delta) \label{equation-ChristoffelDarbouxscaledassociated}
\\ &= b_{n+1}\frac{p_{n+1}(x)
p_{n-m}(y,m,\gamma,\delta)-p_{n-m+1}(y,m,\gamma,\delta)p_n(x)}{x-y} \notag
\\ & \hspace{1cm}+ \frac{p_m(x)((\delta-1) y-\gamma)}{x-y}+ b_m \frac{p_{m-1}(x)}{x-y}. \notag
\end{align}
\end{Lem}

\begin{proof}
We follow the lines of the proof of the original Christoffel-Darboux formula (see \cite[Theorem 4.5]{Chihara}).
By (\ref{equation-recursionorthonormal}) and (\ref{equation-recursionassociatedsymmetric}), we have for $k \geq m$ the identities
\begin{align*}
x p_k&(x)p_{k-m}(y,m)  \\&= b_{k+1} p_{k+1}(x)p_{k-m}(y,m)+
a_k p_k(x) p_{k-m}(y,m) + b_k p_{k-1}(x) p_{k-m}(y,m),\\
y p_k&(x)p_{k-m}(y,m)  \\&= b_{k+1} p_{k}(x)p_{k-m+1}(y,m)+
a_k p_k(x) p_{k-m}(y,m) + b_k p_{k}(x) p_{k-m-1}(y,m).
\end{align*}
Subtracting the second equation from the first, we get
\begin{align*}
(x-y)& p_k(x)p_{k-m}(y,m) \\
& = b_{k+1}\big(p_{k+1}(x)p_{k-m}(y,m) - p_{k}(x)p_{k-m+1}(y,m)\big) \\
& \quad - b_{k} \big(p_{k}(x)p_{k-m-1}(y,m) - p_{k-1}(x)p_{k-m}(y,m)\big).
\end{align*}
Let
\[F_k(x,y) = b_{k+1} \frac{p_{k+1}(x)p_{k-m}(y,m) - p_{k}(x)p_{k-m+1}(y,m)}{x-y}. \]
Then, the last equation can be rewritten as
\[ p_k(x)p_{k-m}(y,m)= F_k(x,y) - F_{k-1}(x,y), \quad k \geq m, \]
where $F_{m-1}(x,y) = - b_m p_{m-1}(x)$.
Summing the latter from $m$ to $n$, we obtain (\ref{equation-ChristoffelDarbouxassociated}). \\
Analogously, we get for the scaled co-recursive associated polynomials
\begin{align*}
p_k(x)p_{k-m}(y,m,\gamma,\delta) &= G_k(x,y) - G_{k-1}(x,y), \quad k \geq m+1, \\
p_m(x)p_{0}(y,m,\gamma,\delta) &= p_m(x), \\
\end{align*}
where
\begin{align*}
G_k(x,y) &= b_{k+1}\frac{p_{k+1}(x)p_{k-m}(y,m,\gamma,\delta) -
p_{k}(x)p_{k-m+1}(y,m,\gamma,\delta)}{x-y},  \\
G_m(x,y) &= \frac{b_{m+1}p_{m+1}(x) -
p_{m}(x)(\delta y-a_m - \gamma)}{x-y}, \quad k \geq m+1.
\end{align*}
Then, summing from $m$ to $n$, we get
\begin{align*}
\sum_{k=m}^n   p_k&(x)p_{k-m}(y,m,\gamma,\delta) = \sum_{k=m+1}^n (G_{k}(x,y)-G_{k-1}(x,y)) +
p_m(x) \\ = & G_n(x,y) - \frac{ b_{m+1} p_{m+1}(x) +
p_{m}(x)(\delta y-a_m - \gamma)}{x-y} + \frac{p_m(x)(x-y)}{x-y} \\ =&
G_n(x,y) + \frac{p_m(x)((\delta-1) y-\gamma)}{x-y}+ b_m \frac{p_{m-1}(x)}{x-y}.
\end{align*}
Hence, we obtain formula \eqref{equation-ChristoffelDarbouxscaledassociated}.
\end{proof}

As a direct consequence of the Christoffel-Darboux type formulas in Lemma \ref{Lemma-ChristoffelDarbouxassociated}, we
get the following explicit formulas for the optimal polynomials in Theorem \ref{Theorem-optimalpolynomial}:

\begin{Cor} \label{Corollary-explicitformoptimalpolynomials}
The optimal polynomials $\Ppa$, $\Ppb$ and $\Ppc$ in Theorem \ref{Theorem-optimalpolynomial} have the explicit form
\begin{align*}
\Ppa(x) &= \kappa_1 b_{n+1} \frac{p_{n+1}(x)
p_{n}(\lambda_{n+1})}{x -\lambda_{n+1}},
\displaybreak[0]\\
\Ppb(x) &= \kappa_2  \frac{ b_{n+1}p_{n+1}(x)
p_{n-m}(\lambda_{n-m+1}^m,m)+ b_m p_{m-1}(x )}{x -\lambda_{n-m+1}^{m}}, \displaybreak[0]\\
\Ppc(x) &= \kappa_3 \left( \Rcal(x) + \frac{b_{n+1} p_{n+1}(x)
p_{n-m}(\lambda_{n-m+1}^{\Rcal},m,\gamma_{\Rcal},\delta_{\Rcal})}
{x -\lambda_{n-m+1}^{\Rcal}} \right. \displaybreak[0]
\\ & \hspace{2cm}+ \left. \frac{p_m(x )((\delta_{\Rcal}-1) \lambda_{n-m+1}^{\Rcal}-\gamma_{\Rcal})+ b_m p_{m-1}(x )}
{x -\lambda_{n-m+1}^{\Rcal}}\right),
\end{align*}
where the constants $\kappa_1$, $\kappa_2$, $\kappa_3$ and the roots $\lambda_{n+1}$ $\lambda_{n-m+1}^{m}$ and $\lambda_{n-m+1}^{\Rcal}$ are given as
in Theorem \ref{Theorem-optimalpolynomial}.
\end{Cor}

\section{Optimally space localized polynomials for Jacobi expansions} \label{Section-optimalpolynomials}

In this section, we will see that in the case of the Jacobi polynomials the generalized mean value $\eps(f)$
is related to an uncertainty principle and that the optimal polynomials $\Ppa$, $\Ppb$ and $\Ppc$ minimize the
term $\var_S(f)$ for the position variance of the uncertainty principle.

In the following, the weight function $w = w_{\alpha\beta}$ under consideration is the Jacobi weight function
\[w_{\alpha\beta}(x) = (1-x)^{\alpha}(1+x)^{\beta}, \qquad x \in [-1,1], \quad \alpha, \beta \geq -\frac{1}{2}.\]
The corresponding orthonormal polynomials $p_n^{(\alpha,\beta)}(x)$ are called Jacobi polynomials and satisfy the differential equation $L_{\alpha\beta} p_n^{(\alpha,\beta)} = -n(n + \alpha + \beta + 1) p_n^{(\alpha,\beta)}$,
where the second-order differential operator $L_{\alpha\beta}$ is given as (cf. \cite[Theorem 4.2.2.]{Szegö})
\[ L_{\alpha\beta} f(x) = (1-x^2) \frac{d^2}{dx^2} f(x) + (\beta-\alpha + x (\alpha+\beta+2)) \frac{d}{dx} f(x). \]
The next theorem states a well-known uncertainty principle for functions having an expansion in terms of Jacobi polynomials.

\begin{Thm}\label{Theorem-uncertaintyJacobi}
Let $f\in C^2([-1,1]) \cap L^2([-1,1],w_{\alpha\beta})$ such that $\|f\|_{w_{\alpha\beta}} = 1$. Further, let
\[(\alpha-\beta)+(\alpha+\beta+2)\eps(f) \neq 0. \]
Then, the following uncertainty inequality holds:
\begin{equation} \label{equation-uncertaintyJacobi}
\frac{1-\eps(f)^2}{|\frac{\alpha-\beta}{\alpha+\beta+2}+\eps(f)|^2} \cdot \langle - L_{\alpha\beta} f, f \rangle_{w_{\alpha\beta}} > \frac{(\alpha+\beta+2)^2}{4}.
\end{equation}
The constant $\frac{(\alpha+\beta+2)^2}{4}$ on the right hand side of $\eqref{equation-uncertaintyJacobi}$ is optimal.
\end{Thm}

\begin{Rem}
Theorem \ref{Theorem-uncertaintyJacobi} has been proven for ultraspherical expansions in \cite{RoeslerVoit1997} and was generalized
to the Jacobi case in \cite{LiLiu2003}. Hereby, the notation in \cite{LiLiu2003} differs slightly from the notation above. For a more detailed discussion of Theorem \ref{Theorem-uncertaintyJacobi} see also \cite{Erb2010}, \cite{ErbDiss}, \cite{GohGoodman2004} and \cite{Selig2002}.
\end{Rem}

The terms
\begin{align} \label{equation-positionvarianceJacobi}
\var_{S}(f) &:= \frac{1-\eps(f)^2}{\big(\frac{\alpha-\beta}{\alpha +\beta+2}+ \eps(f)\big)^2}, \\
\var_{F}(f) &:= \langle - L_{\alpha\beta} f, f \rangle_{w_{\alpha\beta}}
\end{align}
in inequality \eqref{equation-uncertaintyJacobi} are called the position and the frequency variance of the function $f$, respectively.
As the generalized mean value $\eps(f)$, also the position variance $\var_{S}(f)$ defines a
measure for the localization of the function $f$ at the boundary points of the interval $[-1,1]$. In particular, the more mass of
the $L^2$-density $f$ is concentrated at the boundary points, the smaller the position variance $\var_S(f)$ gets. The next Theorem shows that both measures are in principle equivalent.
The only thing one has to take account of is that, in contrast to $\eps(f)$, the position variance $\var_S(f)$ does
not differ between the two boundary points. Therefore, one has to restrict the set of admissible functions in the optimization problem:
\begin{align*}
\Loc_{n} &:= \{P \in \usa: \;\eps(P) > \lambda_1 \}, \\
\Loc_{n}^m &:= \{P \in \usb: \;\eps(P) > \lambda_1 \}, \\
\Loc_n^{\Rcal} &:= \{P \in \usc: \;\eps(P) > \lambda_1 \}, \\
\end{align*}
where $\lambda_1 = \frac{\beta-\alpha}{2+\alpha+\beta}$ corresponds to the sole root of the
Jacobi polynomial $p_1^{(\alpha,\beta)}(x)$ of degree $1$.

\begin{Thm} \label{Theorem-minequivalenttomax} If the sets $\Loc_{n}$, $\Loc_{n}^{m}$ and $\Loc_{n}^{\Rcal}$ are nonempty, then
\begin{align*}
\arg\min_{P \in \Loc_{n}} \var_{S}(P) &= \arg\max_{P \in \Loc_{n}} \eps(P) = \Ppa,\\
\arg\min_{P \in \Loc_{n}^m} \var_{S}(P) &= \arg\max_{P \in \Loc_{n}^m} \eps(P) = \Ppb,\\
\arg\min_{P \in \Loc_n^{\Rcal}} \var_{S}(P) &= \arg\max_{P \in \Loc_n^{\Rcal}} \eps(P) = \Ppc.
\end{align*}
Hence, from all polynomials in the sets $\Loc_{n}$, $\Loc_{n}^m$, $\Loc_{n}^\Rcal$, the optimal polynomials $\Ppa$, $\Ppb$, $\Ppc$ minimize the
position variance $\var_S$.
\end{Thm}

\begin{proof}
We consider the space variance $\var_{S}$ as a function of $\lambda = \eps(f)$. We have
\begin{align*}
\var_{S}(\lambda) &= \frac{1-\lambda^2}{(\lambda-\lambda_1)^2},\\
\frac{d \var_{S}}{d\lambda} (\lambda) &= \frac{-2(\lambda-\lambda_1)\lambda-2(1-\lambda^2)}{(\lambda-\lambda_1)^3}
= \frac{-2(1-\lambda_1 \lambda)}{(\lambda-\lambda_1)^3}.
\end{align*}
Therefore, the derivative $\frac{d}{d\lambda} \var_{S}$ is strictly decreasing on the open interval $(\lambda_1,1)$ and
strictly increasing on $(-1,\lambda_1)$.
So, for $P \in \Loc_{n}, \Loc_{n}^m,\Loc_n^{\Rcal}$, maximizing $\eps(P)$
yields the same result as minimizing the position variance $\var_{S}(P)$.
\end{proof}

\begin{Rem} \label{Remark-nonemptinessofsetsJacobi}
Whereas it can not be guaranteed that the sets $\Loc_{n}^m$ and $\Loc_n^{\Rcal}$
are nonempty, the non-emptiness of the sets $\Loc_{n}$, $n \geq 1$, is a consequence of
the interlacing property of the zeros
of the Jacobi polynomials (cf. \cite[Theorem 3.3.2]{Szegö}, \cite[Theorem 5.3]{Chihara}). Namely, this interlacing property implies that
$\eps(\Ppa) = \lambda_{n+1} > \lambda_n > \ldots > \lambda_1$.
\end{Rem}

\begin{Rem}
It can be shown that the uncertainty product $\var_S(\Ppa) \cdot \var_F(\Ppa)$ in \eqref{equation-uncertaintyJacobi} of the optimal polynomials $\Ppa$ is uniformly bounded by a constant independent of the degree $n$. Hence, the polynomials $\Ppa$ are not only well localized in space, but also in space and frequency. The quite technical proof can be found in \cite{ErbDiss}.
\end{Rem}

\begin{Exa} \label{example-optimallocalizedchebyshev}
As a final example, we consider the orthonormal Chebyshev polynomials $t_n$ of first kind, i.e., the Jacobi polynomials $p_n^{(\alpha,\beta)}$ with $\alpha = \beta = -\frac{1}{2}$ and the weight function $w_{\alpha\beta}(x) = 1$. The orthonormal Chebyshev polynomials are explicitly given as
(see \cite[p.~28-29]{Gautschi})
\[ t_0(x) = \ts \frac{1}{\sqrt{\pi}}, \quad t_n(x) = \ts \sqrt{\frac{2}{\pi}} \cos n t, \quad n \geq 1, \]
where $x = \cos t$. The largest zero of the Chebyshev polynomials $t_{n+1}$ is given by $\lambda_{n+1} = \cos \frac{\pi}{2n+2}$ (see \cite[(6.3.5)]{Szegö}). The normalized
associated polynomials $t_n(x,m)$, $m \geq 1$, correspond to the Chebyshev polynomials $u_n$ of the second kind given by (see \cite[p.~28-29]{Gautschi})
\[ u_n(x) =  \sqrt{\frac{2}{\pi}} \frac{\sin (n+1) t}{\sin t}, \quad n \geq 0.\]
The largest zero of the polynomials $u_{n+1}$ is given by $\lambda_{n+1} = \cos \frac{\pi}{n+2}$.
So, in the case of the Chebyshev polynomials of first kind, we get for the optimally space localized polynomials the formulas
\begin{align}
\mathcal{T}_n(x) &= \frac{\kappa_1}{\pi} \left( 1 + 2 \sum_{k=1}^n \cos\frac{k \pi }{2n+2} \cos kt\right) =
\frac{\kappa_1}{\pi} \frac{\cos (n+1)t \, \cos \frac{n\pi}{2n+2}}{\cos t -\cos\frac{\pi}{2n+2}}. \label{equation-optimalchebyshev}\\
\mathcal{T}_n^m (x)
&= \frac{2\kappa_2}{\pi } \left(\sum_{k=m}^n \frac{\sin\frac{(k-m+1) \pi }{n-m+2} }{\sin\frac{\pi }{n-m+2}}
 \cos kt \right) = \frac{2\kappa_2}{\pi}\frac{\cos (\frac{n-m+2 }{2} t) \cos (\frac{n+m }{2} t)}{ \cos t - \cos \frac{\pi }{n-m+2}}.
\label{equation-optimalchebyshevwavelet}
\end{align}
These optimal polynomials $\mathcal{T}_n$ and $\mathcal{T}_n^m(x)$ in combination with the Breitenberger uncertainty principle on the unit circle were intensively studied by Rauhut et al. in \cite{PrestinQuakRauhutSelig2003} and \cite{Rauhut2005}.
\end{Exa}

\section{Optimally space localized polynomials for Hermite expansions}

As an example of an orthogonal expansion in a non-compact setting we consider the Hermite polynomials on the real line. The aim of this section is to construct polynomials having an expansion in terms of the Hermite polynomials that are optimally localized at the point $x = 0$. In the following, we will see that most of the theory of the previous sections can be applied also in this case, although with slight modifications.

The Hilbert space under consideration is now $L^2(\Rr,w_H)$ with the weight function  $w_H(x) = e^{-x^2}$.
The corresponding orthonormal polynomials $(h_l)_{l=0}^\infty$, defining an orthonormal basis of $L^2(\Rr,w_H)$,
are called the (orthonormal) Hermite polynomials on $\Rr$.

As in Definition \eqref{definition-polynomialspacesJacobi}, we introduce
the polynomial spaces $\pia$, $\pib$ and $\pic$, and the corresponding unit spheres $\usa$, $\usb$ and $\usc$ for the Hermite polynomials $h_n$. The goal is, similar as in Section \ref{Section-optimallylocalizedpolynomials}, to find those polynomials from $\usa$, $\usb$ and $\usc$ that minimize the position variance
\begin{equation}
\var_S(f) := \int_{\Rr} x^2 |f(x)|^2 e^{-x^2} dx.
\end{equation}
Since in the setting of the Hermite polynomials the calculations will be more complex, we will omit the case
$P \in \pic$. Also, we will assume that $n$ and $m$ are even integers. The minimization problems then read as follows:
\begin{align}
 \Hha &= \arg\min_{P \in \usa} \var_S(P),  \label{optimalhermitea}\\
 \Hhb &= \arg\min_{P \in \usb} \var_S(P).  \label{optimalhermiteb}
\end{align}

To get the explicit solutions, we will make use of the Laguerre polynomials $p_l^{(\alpha)}$ that form an
orthonormal basis of the Hilbert space $L^2([0,\infty),w_{\alpha})$ with the weight function $w_{\alpha}(x) = x^{\alpha} e^{-x}$,
$\alpha > -1$. The orthonormal Laguerre polynomials $p_l^{(\alpha)}$ and the Hermite polynomials $h_l$ are correlated by the
following two formulas (see \cite[Section 4.6]{Ismail}):
\begin{align}
h_{2l}(x) &=  p_l^{(-1/2)}(x^2), & l = 0,1,2, \ldots, \label{equation-correlationhermiteLaguerreeven}\\
h_{2l+1}(x) &= x \, p_l^{(1/2)}(x^2), & l= 0,1,2, \ldots. \label{equation-correlationhermiteLaguerreodd}
\end{align}
The next Lemma gives a characterization of the position variance $\var_S(P)$ in terms of the expansions coefficients of a polynomial
$P$.

\begin{Lem} \label{Lemma-characterizationofvarS}
For a polynomial $P(x) = \ds \sum_{l=0}^n c_l h_l(x)$, we get the formulas
\begin{align*}
\var_S(P) &= \ts \ccc_e^H \Jj(-\frac{1}{2})_{\frac{n}{2}} \ccc_e + \ccc_o^H \Jj(\frac{1}{2})_{\frac{n}{2}-1} \ccc_o,  & \text{if} \quad P \in \pia, \\
\var_S(P) &= \ts \tilde{\ccc}_e^H \Jj(-\frac{1}{2})_{\frac{n}{2}}^{\frac{m}{2}} \tilde{\ccc}_e + \tilde{\ccc}_o^H \Jj(\frac{1}{2})_{\frac{n}{2}-1}^{\frac{m}{2}} \tilde{\ccc}_o, & \text{if} \quad P \in \pib,
\end{align*}
with the coefficient vectors
\begin{align*}
\ccc_e &= (c_0, c_2, \ldots, c_n)^T, & \ccc_o &= (c_1, c_3, \ldots, c_{n-1})^T, \\
\tilde{\ccc}_e &= (c_m, c_{m+2}, \ldots, c_n)^T, & \tilde{\ccc}_o &= (c_{m+1}, c_{m+3}, \ldots, c_{n-1})^T,
\end{align*}
and the matrices $\Jj(-\frac{1}{2})_{\frac{n}{2}}^{\frac{m}{2}}$ and $\Jj(\frac{1}{2})_{\frac{n}{2}-1}^{\frac{m}{2}}$ corresponding to the Jacobi matrices of the associated Laguerre polynomials $p_l^{(-\frac{1}{2})}(x,\frac{m}{2})$ and $p_l^{(\frac{1}{2})}(x,\frac{m}{2})$, respectively.
\end{Lem}

\begin{proof}
Using the correlations \eqref{equation-correlationhermiteLaguerreeven} and \eqref{equation-correlationhermiteLaguerreodd}
between the Hermite and Laguerre polynomials as well as the three-term recurrence formulas (\ref{equation-recursionorthonormal}) of
the Laguerre polynomials $p_l^{(-\frac{1}{2})}(x)$ and $p_l^{(\frac{1}{2})}(x)$, we get for the polynomial $P(x) = \ds \sum_{l=0}^n c_l h_l(x)$ the formula
\begin{align*}
\var_S& (P) = \int_\Rr x^2 \Big|\sum_{l=0}^n c_l h_l(x)\Big|^2 e^{-x^2} dx \\ &= \int_\Rr \Big(\sum_{l=0}^{\frac{n}{2}} c_{2l} x^2  p_l^{(- \frac{1}{2})}(x^2)+ \sum_{l=0}^{\frac{n}{2}-1} c_{2l+1} x^3  p_l^{(\frac{1}{2})}(x^2)\Big)
\overline{\Big(\sum_{l=0}^n c_l h_l(x)\Big)} e^{-x^2} dx \\
&= \int_\Rr \Big(\sum_{l=0}^{\frac{n}{2}} c_{2l} \ts \big(b_{l+1}^{(-\frac{1}{2})} p_{l+1}^{(-\frac{1}{2})}(x^2)
+ a_l^{(-\frac{1}{2})} p_l^{(-\frac{1}{2})}(x^2)+ b_{l}^{(-\frac{1}{2})} p_{l-1}^{(- \frac{1}{2})}(x^2) \big) \\ & \quad + \sum_{l=0}^{\frac{n}{2}-1} c_{2l+1}  x \big( \ts b_{l+1}^{(\frac{1}{2})}  p_{l+1}^{(\frac{1}{2})}(x^2)
+ a_l^{(\frac{1}{2})} p_l^{(\frac{1}{2})}(x^2)+ b_{l}^{(\frac{1}{2})} p_{l-1}^{(\frac{1}{2})}(x^2) \big)\Big)
\overline{\Big(\sum_{l=0}^n c_l h_l(x)\Big)} e^{-x^2} dx  \\
&= \int_\Rr \Big(\sum_{l=0}^{\frac{n}{2}} c_{2l}\ts  \big(b_{l+1}^{(-\frac{1}{2})}  h_{2l+2}(x)
+ a_l^{(-\frac{1}{2})} h_{2l}(x) + b_{l}^{(-\frac{1}{2})} h_{2l-2}(x) \big) \\ & \quad + \sum_{l=0}^{\frac{n}{2}-1} c_{2l+1} \ts \big(b_{l+1}^{(\frac{1}{2})}  h_{2l+3}(x)
+ a_l^{(\frac{1}{2})} h_{2l+1}(x)+ b_{l}^{(\frac{1}{2})} h_{2l-1}(x) \big)\Big)
\overline{\Big(\sum_{l=0}^n c_l h_l(x)\Big)} e^{-x^2} dx.
\end{align*}
Next, using the orthonormality relations of the Hermite polynomials $h_l$, we can conclude for $P \in \pia$
\begin{align*}
\var_S(P) &= \ccc_e^H \ts \Jj(-\frac{1}{2})_{\frac{n}{2}} \ccc_e + \ccc_o^H \Jj(\frac{1}{2})_{\frac{n}{2}-1} \ccc_o.
\end{align*}
If $c_0 = \ldots = c_{m-1} = 0$, we get the assertion for the polynomials $P$ in $\pib$.
\end{proof}

The solutions of the optimization problems \eqref{optimalhermitea} and \eqref{optimalhermiteb} now read as follows.

\begin{Thm} \label{Theorem-optimalhermite}
The polynomials solving the minimization problems \eqref{optimalhermitea} and \eqref{optimalhermiteb} are given by
\begin{align}
\Hha (x) &= \kappa_1 \sum_{l=0}^{\frac{n}{2}} p_l^{(-\frac{1}{2})}(\lambda_{\frac{n}{2}+1})\, h_{2l}(x),
\displaybreak[0] \label{equation-optimalpolynomialhermitea}\\
\Hhb (x) &= \kappa_2 \sum_{l=\frac{m}{2}}^\frac{n}{2} \ts p_{l-\frac{m}{2}}^{(-\frac{1}{2})}(\lambda_{\frac{n-m}{2}+1}^{\frac{m}{2}},\frac{m}{2})\, h_{2l}(x),
\displaybreak[0] \label{equation-optimalpolynomialhermiteb}
\end{align}
where $p_l^{(-\frac{1}{2})}(x,\frac{m}{2})$ denote the associated Laguerre polynomials with parameter $\alpha = -\frac{1}{2}$.
The values $\lambda_{\frac{n}{2}+1}$ and $\lambda_{\frac{n-m}{2}+1}^{\frac{m}{2}}$
denote the smallest zero of the polynomials
$p_{\frac{n}{2}+1}^{(-\frac{1}{2})}(x)$ and $p_{\frac{n-m}{2}+1}^{(-\frac{1}{2})}(x,\frac{m}{2})$, respectively. The constants $\kappa_1$ and $\kappa_2$ are chosen such that the optimal polynomials lie in the respective unit sphere and are uniquely determined up to multiplication with a complex scalar of absolute value one.
The minimal value of $\var_S(P)$ in the respective polynomial spaces is given by
\begin{align*}
M_{n}  &:= \min_{P\in \usa}\var_S(P) = \lambda_{\frac{n}{2}+1},\\
M_{n}^m  &:= \min_{P\in \usb}\var_S(P) = \lambda_{\frac{n-m}{2}+1}^{\frac{m}{2}}.
\end{align*}
\end{Thm}

\begin{proof}
In the following, we will determine the optimal solution $\Hhb$ for
the minimization problem \eqref{optimalhermiteb}. The formula for
the optimal polynomial $\Hha$ follows as a special case when $m=0$. By Lemma \ref{Lemma-characterizationofvarS}, the
position variance $\var_S(P)$ of a polynomial $P(x) = \sum_{l=m}^n c_l h_l(x)$
can be written as $\var_S(P) = \ts \tilde{\ccc}_e^H \Jj(-\frac{1}{2})_{\frac{n}{2}}^{\frac{m}{2}} \tilde{\ccc}_e + \tilde{\ccc}_o^H \Jj(\frac{1}{2})_{\frac{n}{2}-1}^{\frac{m}{2}} \tilde{\ccc}_o$ with the coefficient vectors $\tilde{\ccc}_e$ and $\tilde{\ccc}_e$ given in Lemma \ref{Lemma-characterizationofvarS}.
Hence, minimizing $\var_S(P)$ with respect to a normed polynomial $P \in \usb$
is equivalent to minimize the quadratic functional
\begin{equation} \ts \tilde{\ccc}_e^H \Jj(-\frac{1}{2})_{\frac{n}{2}}^{\frac{m}{2}} \tilde{\ccc}_e + \tilde{\ccc}_o^H \Jj(\frac{1}{2})_{\frac{n}{2}-1}^{\frac{m}{2}} \tilde{\ccc}_o \quad \text{subject to} \quad
|\tilde{\ccc}_e|^2+|\tilde{\ccc}_o|^2 = c_m^2 + c_{m+1}^2 + \cdots + c_n^2 = 1. \end{equation} \label{equation-minimizationproblemcoefficients}
The minimization problem \eqref{equation-minimizationproblemcoefficients} has a block matrix structure with the two
matrices $\Jj(-\frac{1}{2})_{\frac{n}{2}}^{\frac{m}{2}}$ and $\Jj(\frac{1}{2})_{\frac{n}{2}-1}^{\frac{m}{2}}$. To solve the problem, we have to determine the smallest eigenvalue of both matrices. The eigenvalues of the two matrices $\Jj(-\frac{1}{2})_{\frac{n}{2}}^{\frac{m}{2}}$ and $\Jj(\frac{1}{2})_{\frac{n}{2}-1}^{\frac{m}{2}}$ correspond exactly with the zeros of the Laguerre polynomials $p_{\frac{n-m}{2}+1}^{(-\frac{1}{2})}(x,\frac{m}{2})$ and $p_{\frac{n-m}{2}}^{(\frac{1}{2})}(x,\frac{m}{2})$. By a functional analytic method based on the Hellmann-Feynman Theorem (see \cite{Ismail1987}, \cite[Section 7.4]{Ismail} and \cite{ErbTookos2010}) it follows that the smallest zero of the associated Laguerre polynomials $p_l^{(\alpha)}(x,m)$ is an increasing function of the parameter $\alpha$. This result together with the fact that the smallest eigenvalue of $p_{\frac{n-m}{2}+1}^{(-\frac{1}{2})}(x,\frac{m}{2})$ is strictly smaller than the one of $p_{\frac{n-m}{2}}^{(-\frac{1}{2})}(x,\frac{m}{2})$ (see the interlacing property of the orthogonal polynomials \cite[Theorem 3.3.2]{Szegö}) implies that the matrix $\Jj(-\frac{1}{2})_{\frac{n}{2}}^{\frac{m}{2}}$ is the one with the smallest eigenvalue.

Therefore, if $\lambda_{\frac{n-m}{2}+1}^{\frac{m}{2}}$ denotes the smallest eigenvalue of the Jacobi
matrix $\Jj(-\frac{1}{2})_{\frac{n}{2}}^{\frac{m}{2}}$, we get
\begin{equation} \label{equation-extremalmatrixHermite}
\tilde{\ccc}_e^H \Jj(-\frac{1}{2})_{\frac{n}{2}}^{\frac{m}{2}} \tilde{\ccc}_e + \tilde{\ccc}_o^H \Jj(\frac{1}{2})_{\frac{n}{2}-1}^{\frac{m}{2}} \tilde{\ccc}_o \geq \lambda_{\frac{n-m}{2}+1}^{\frac{m}{2}} (|\tilde{\ccc}_e|^2+ |\tilde{\ccc}_o|^2)
\end{equation}
and equality is attained for the eigenvectors corresponding to $\lambda_{\frac{n-m}{2}+1}^{\frac{m}{2}}$.
Using the recursion formula (\ref{equation-recursionassociatedsymmetric}) of the associated polynomials $p_{l}^{(-\frac{1}{2})}(x,\frac{m}{2})$ with
$c_m = 1$ the eigenvalue equation $\Jj(-\frac{1}{2})_{\frac{n}{2}}^{\frac{m}{2}} \tilde{\ccc}_e = \lambda_{{\frac{n-m}{2}}+1}^{{\frac{m}{2}}} \tilde{\ccc_e}$ yields
\begin{align*}
c_{2l} &= \ts p_{l-\frac{m}{2}}(\lambda_{\frac{n-m}{2}+1}^{\frac{m}{2}},\frac{m}{2}), & \ts l = \frac{m}{2}, \ldots, \frac{n}{2},\\
c_{2l+1} &= 0, & \ts l = \frac{m}{2}, \ldots, \frac{n}{2}-1.
\end{align*}
Finally, we have to normalize the coefficients
$c_l$, $m \leq l \leq n$, such that $|\tilde{\ccc}_e|^2+|\tilde{\ccc}_o|^2 =1$.
This is done by the absolute value of the constant $\kappa_2$.
The uniqueness (up to a complex scalar with absolute value $1$) of the optimal polynomial $\Ppa$
follows from the fact that the smallest zero
of the polynomial $p_{\frac{n-m}{2}+1}^{(-\frac{1}{2})}(x,\frac{m}{2})$ is simple (see \cite[Theorem 5.3]{Chihara}).
The formula for $M_{n}^m$ follows directly from the estimate in (\ref{equation-extremalmatrixHermite}).
\end{proof}

Using once again the relation \eqref{equation-correlationhermiteLaguerreeven} between the even Hermite polynomials and
the Laguerre polynomials and the Christoffel-Darboux type formulas ot Lemma \ref{Lemma-ChristoffelDarbouxassociated}, we
get the following explicit formulas for the optimal polynomials in Theorem \ref{Theorem-optimalhermite}:

\begin{Cor} \label{Corollary-explicitformoptimalhermite}
The optimal polynomials $\Hha$ and $\Hhb$ in Theorem \ref{Theorem-optimalhermite} have the following explicit form:
\begin{align*}
\Hha(x) &= \kappa_1 b_{n+1}^{(-\frac{1}{2})} p_{\frac{n}{2}}^{(-\frac{1}{2})}(\lambda_{\frac{n}{2}+1}) \frac{h_{n+1}(x)}{x^2 - \lambda_{\frac{n}{2}+1}},
\displaybreak[0]\\
\Hhb(x) &= \kappa_2  \frac{ b_{n+1}^{(-\frac{1}{2})} p_{\frac{n-m}{2}}^{(-\frac{1}{2})}(\lambda_{\frac{n-m}{2}+1}^{\frac{m}{2}},\frac{m}{2}) h_{n+1}(x)
+ b_m^{(-\frac{1}{2})} h_{m-1}(x)}{x^2 -\lambda_{\frac{n-m}{2}+1}^{\frac{m}{2}}}, \displaybreak[0]
\end{align*}
where the constants $\kappa_1$, $\kappa_2$ and the roots $\lambda_{\frac{n}{2}+1}$, $\lambda_{{\frac{n-m}{2}}+1}^{{\frac{m}{2}}}$ are given as in Theorem \ref{Theorem-optimalhermite}.
\end{Cor}

\section{Construction of polynomial filters for the detection of peaks in periodic signals}

In this final section, we will give some examples on how the optimal polynomials of Section \ref{Section-optimalpolynomials} can be applied as filters for the detection of peaks. To this end, we consider continuous $2\pi$-periodic signal functions $f \in L^2([-\pi,\pi))$ and trigonometric polynomial filters $h \in \Pi_n$.

Our goal is to find trigonometric polynomials $h$ which are well suited to work out the peaks of the signal $f$. Since the
filtering operator $F_h$ defined in \eqref{equation-filteroperator} acts as a convolution operator on $f$, most of the mass of the polynomial $h$ has to be concentrated at the point $t = 0$ in order to filter out the peaks of the signal $f$. If we further assume that  $h$ is even, i.e., $h(t) = h(-t)$, then $h(\arccos(x))$ is defined on the interval $[-1,1]$ and is a polynomial of degree $n$ in the variable $x = \cos t$. Moreover, if the polynomial $h(\arccos(x))$ is localized at $x = 1$, then the trigonometric polynomial $h$ is localized at $t = 0$. Therefore, the optimally space localized polynomials $\Ppa(\cos t)$ of Theorem \ref{Theorem-optimalpolynomial} are natural choices for polynomial filters in peak analysis.

Experimenting with different weight functions $w$ in Theorem \ref{Theorem-optimalpolynomial} and Corollary \ref{Corollary-explicitformoptimalpolynomials}, it is possible to construct a whole bunch of well-localized polynomial filters with different properties. We give here just some easy examples using the Jacobi weight function $w_{\alpha\beta}$.
For $\alpha = -\frac{1}{2}$, $\beta = -\frac{1}{2}$, and $\alpha = \frac{1}{2}$, $\beta = -\frac{1}{2}$, we get the two filter kernels
\begin{align} \label{equation-optimalfilter1}
h_n^{(1)}(t) &:= \mathcal{T}_n(\cos t ) =  C_n^{(1)} \frac{\cos (n+1)t}{\cos t -\cos \frac{\pi}{2n+2}}, \\
h_n^{(2)}(t) &:=  C_n^{(2)} \frac{\sin (n+\frac{3}{2})t}{\sin \frac{t}{2}} \frac{1}{\cos t -\cos \frac{\pi}{n+\frac{3}{2}}}, \label{equation-optimalfilter2}
\end{align}
where the constants $C_n^{(1)}$ and $C_n^{(2)}$ denote normalizing factors such that the respective polynomials are normed in the $L^2$-norm. The optimal polynomial $h_n^{(1)}(t)$ was already computed in Example \ref{example-optimallocalizedchebyshev}. In approximation theory, the trigonometric polynomials $h_n^{(1)}$ are known as Rogosinski kernels (cf. \cite[p. 112-114]{Lasser}), in signal analysis they are well-known as cosine windows (cf. \cite{Harris1978}).

\begin{figure}[h]
  \begin{minipage}{0.5\textwidth}
   \centering
  The Rogosinski filter $h^{(1)}_6$. \\
  \includegraphics[width=\textwidth]{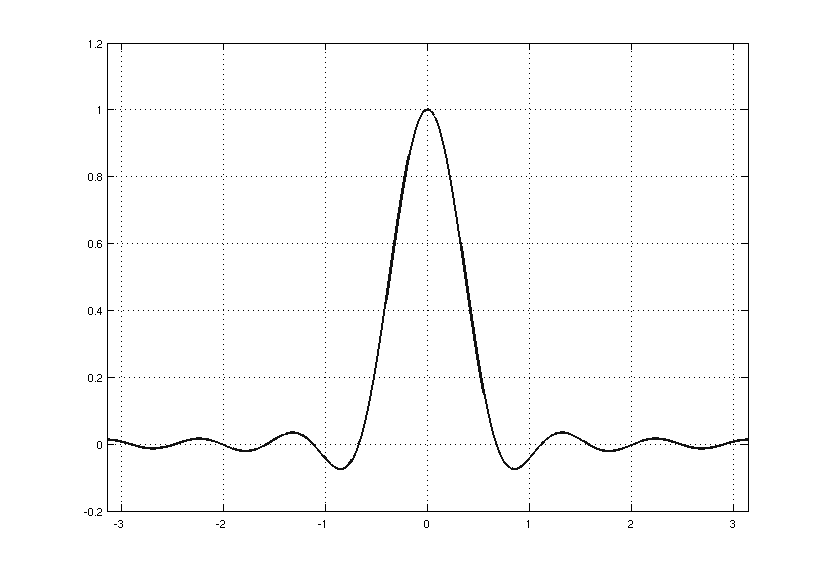}
  \end{minipage}\hfill
  \begin{minipage}{0.5\textwidth}
  \centering
  The polynomial filter $h^{(2)}_6$ \\
  \includegraphics[width=\textwidth]{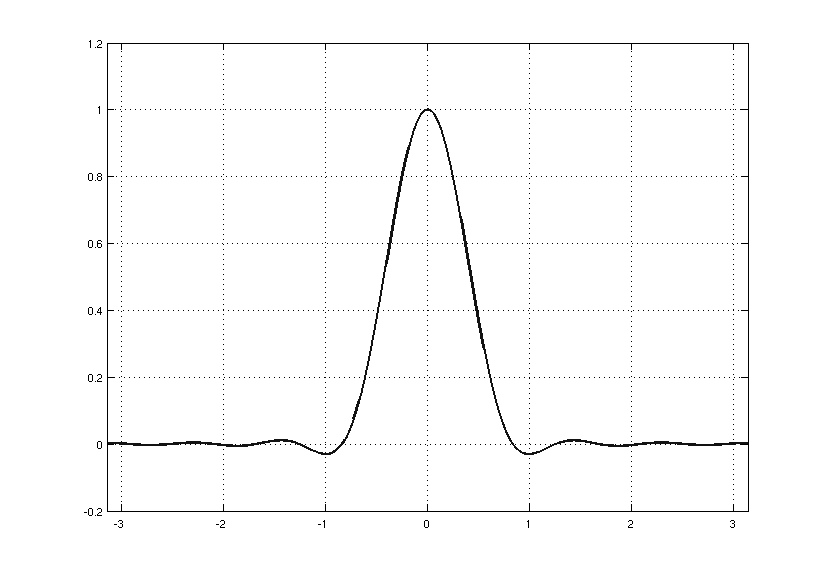}\\
  \end{minipage}
\caption{The polynomial filters $h^{(1)}_6$ and $h^{(2)}_6$ of degree $6$.}
\label{Figure-optimalcosine}
\end{figure}

The filter $h_n^{(2)}(t)$ is computed in the same way as $h_n^{(1)}(t)$ using the explicit representation of the Chebyshev polynomials of
third kind (for the definition, see \cite[Section 1.5.1]{Gautschi}). Figure \ref{Figure-optimalcosine} illustrates that compared to the filter $h_n^{(1)}$ the trigonometric polynomial $h_n^{(2)}$ has a wider peak at $t = 0$ but less mass at the ends $t = \pi$ and $t = -\pi$. This is due to the fact that in the case of the filter $h_n^{(2)}$ we optimize over all polynomials $P \in \Pi_n$ with
\[ \int_{-1}^1 |P(x)|^2 (1-x)^{\frac{1}{2}}(1+x)^{-\frac{1}{2}} dx = 1, \]
i.e. the particular optimization problem favours polynomials that have more mass concentrated at $x = 1$.

\begin{figure}[ht]
\centering
\includegraphics[width=\textwidth]{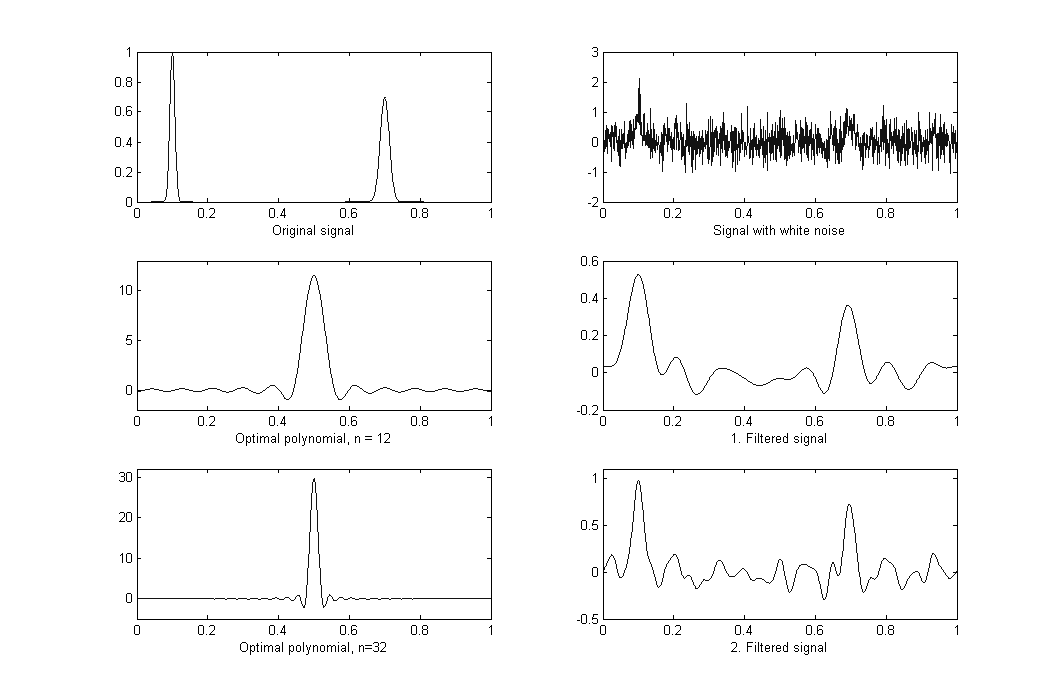}\\
\caption{Filtering a noisy signal with the optimal polynomial filter $h_n^{(1)}$.}
\label{Figure-peakdetectionOP}
\end{figure}

If the peaks of the signal $f$ lie on a low-frequency carrier signal or if some baseline correction has to be done, it is
reasonable to additionally filter out the low frequencies of $f$. In this case, the polynomial filter $h$ has to be restricted to a frequency band $[m,n] \subset \Nn$. The corresponding optimal filters are given in equation \eqref{equation-optimalpolynomialJacobib}.
In case of the Chebyshev polynomials $t_l(x)$ of first kind, the optimal band-limited polynomials $\mathcal{T}_n^{m}$ are given in Example \ref{example-optimallocalizedchebyshev} as
\begin{equation}
h_{n,m}^{(1)}(t) := \mathcal{T}_{n}^{m}(\cos t) = C_{n,m}^{(1)} \frac{\cos (\frac{n-m+2 }{2} t) \cos (\frac{n+m }{2} t)}{ \cos t - \cos \frac{\pi }{n-m+2}}.
\end{equation}

\begin{figure}[ht]
\centering
\includegraphics[width=\textwidth]{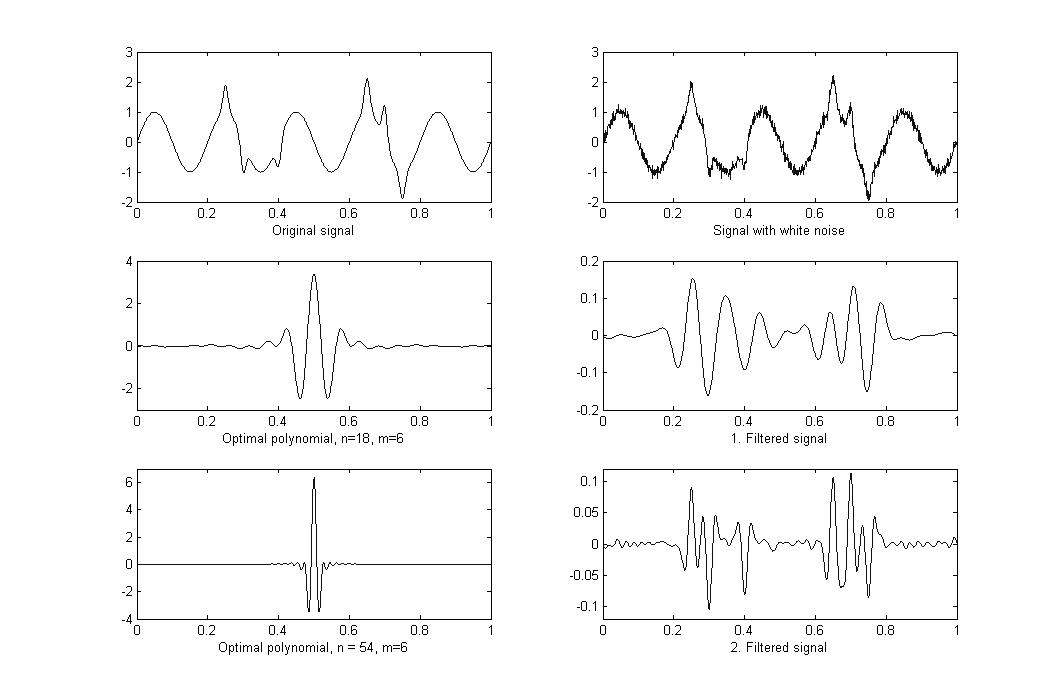}\\
\caption{Filtering a noisy signal with the optimal wavelet filter $h_{n,m}^{(2)}$.}
\label{Figure-peakdetectionOPw}
\end{figure}

Finally, we consider polynomial window functions $h$ with the additional property $h(\pi) = h(-\pi) = 0$. If $h$ is assumed to be a symmetric trigonometric polynomial, then $h$ has the form $h(t) = g(t) (1+\cos t)$,
where $g$ is a symmetric trigonometric polynomial of degree $n-1$. Now, as above, we can play the same game with the polynomial $g$,
i.e., we can insert optimally space localized polynomials as candidates for $g$. Using Jacobi weights with $\alpha = -\frac{1}{2}$,
$\beta = \frac{3}{2}$ and $\alpha = \frac{1}{2}$, $\beta = \frac{1}{2}$, we can introduce the filter kernels
\begin{align} \label{equation-optimalfilter3}
h_n^{(3)}(t) &:=  C_n^{(3)} \frac{p_{n}^{(-\frac{1}{2},\frac{3}{2})}(\cos t) (1+\cos t) }{\cos t -\lambda_{n}}, \\
h_n^{(4)}(t) &:=  C_n^{(4)} \frac{\sin (n+1)t}{\sin t} \frac{1+\cos t}{\cos t -\cos \frac{\pi}{n+1}}, \label{equation-optimalfilter4}
\end{align}
where $\lambda_n$ denotes the largest zero of the Jacobi polynomial $p_{n}^{(-\frac{1}{2},\frac{3}{2})}$.
The filter $h_n^{(3)}$ is defined such that the functional $\int_{-1}^1 x |P(x)|^2 (1+x)^2 dx$ is maximized over all polynomials $P$
under the constraint $\int_{-1}^1 |P(x)|^2 (1+x)^2 dx = 1$. The filter polynomial $h_n^{(4)}$ is well-known in signal analysis under
the name Hann window (see \cite{Harris1978}). In fact,
computing the Fourier coefficients of $h_n^{(4)}$, one gets $\hat{h}_n^{(4)}(j) = c \cos^2(\frac{\pi j}{2n+2})$ for $j = -n, \ldots,n$.

\begin{figure}[h]
  \begin{minipage}{0.5\textwidth}
   \centering
  The polynomial window $h^{(3)}_6$. \\
  \includegraphics[width=\textwidth]{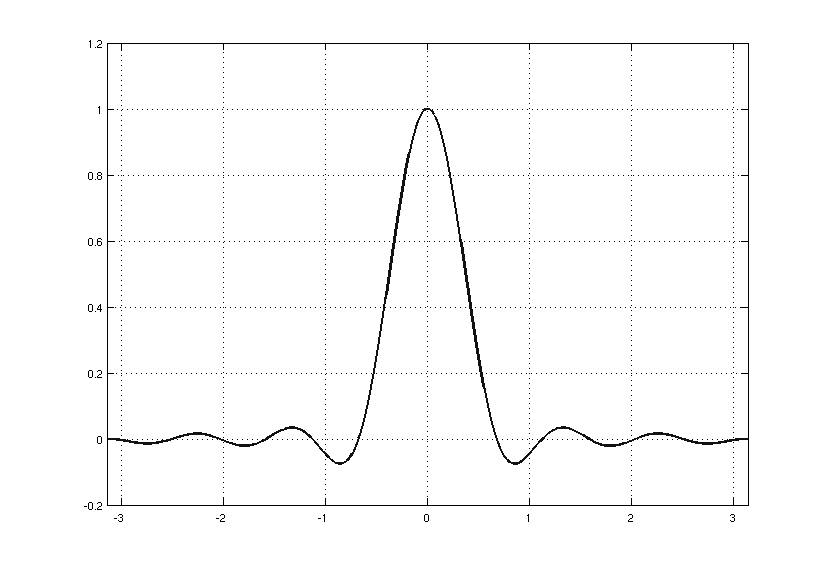}
  \end{minipage}\hfill
  \begin{minipage}{0.5\textwidth}
  \centering
  The Hann window $h^{(4)}_6$ \\
  \includegraphics[width=\textwidth]{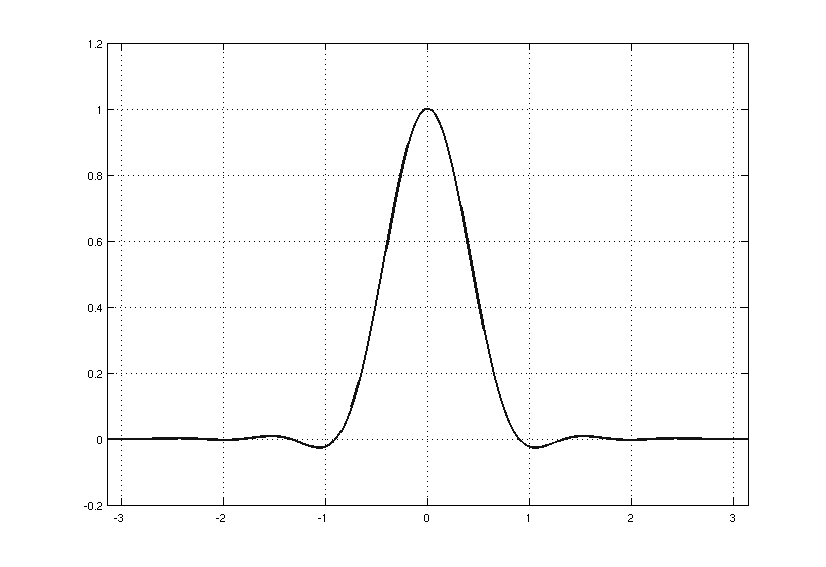}\\
  \end{minipage}
\caption{The polynomial filters $h^{(3)}_6$ and $h^{(4)}_6$ of degree $6$.}
\label{Figure-optimalHann}
\end{figure}

\end{document}